\documentclass{amsart}
\usepackage{graphicx}
\usepackage{ulem}
\usepackage{bm}
\usepackage{color}
\usepackage{amssymb, amsmath, amsthm}
\usepackage{bm}
\theoremstyle{plain}
\usepackage{graphicx}
\usepackage[aboveskip=-5pt,position=bottom]{caption}
\usepackage{natbib}
\usepackage{url}
\usepackage{enumerate}
\usepackage{rotating}

\newtheorem{lemma}{Lemma}[section]

\newtheorem{theorem}[lemma]{Theorem}
\newtheorem{cor}[lemma]{Corollary}

\newtheorem{exam}[lemma]{\normalfont \scshape
 Example}
\newtheorem{rem}[lemma]{\normalfont \scshape Remark}

\newcommand{\R}{\mathbb{R}}
\newcommand{\N}{\mathbb{N}}

\newcommand{\norm}[1]{\left\Vert#1\right\Vert}
\newcommand{\abs}[1]{\left\vert#1\right\vert}
\newcommand{\set}[1]{\left\{#1\right\}}

\newcommand{\bfzero}{\bm{0}}

\newcommand{\bfs}{\bm{s}}

\newcommand{\bft}{\bm{t}}

\newcommand{\bfX}{\bm{X}}

\newcommand{\bfZ}{\bm{Z}}

\newcommand{\bfeta}{\bm{\eta}}

\newcommand{\bfxi}{\bm{\xi}}

\DeclareMathOperator{\Cov}{Cov}
\DeclareMathOperator{\Var}{Var}
\DeclareMathOperator{\MSE}{MSE}

\usepackage[doublespacing]{setspace}

\allowdisplaybreaks[4]
\begin{document}

\title[Generalized max-linear models]{On generalized max-linear models and their statistical interpolation}%

\author{Michael Falk, Martin Hofmann, Maximilian Zott}
\address{University of Wurzburg,
Institute of Mathematics,  Emil-Fischer-Str. 30, 97074 W\"{u}rzburg, Germany.}
\email{michael.falk@uni-wuerzburg.de, maximilian.zott@uni-wuerzburg.de,\newline hofmann.martin@mathematik.uni-wuerzburg.de}

%\thanks{}%

%\date{\today}%
%\dedicatory{}%
%\commby{}%
% ----------------------------------------------------------------
\begin{abstract}
We propose a way how to generate a max-stable process in $C[0,1]$ from a max-stable random vector in $\R^d$ by generalizing the \emph{max-linear model} established by \citet{wansto11}. It turns out that if the random vector follows some finite dimensional distribution of some initial max-stable process, the approximating processes converge uniformly to the original process and the pointwise  mean squared error can be represented in a closed form. The obtained results carry over to the case of generalized Pareto processes. The introduced method enables the reconstruction of the initial process only from a finite set of observation points and, thus, reasonable prediction of max-stable processes in space becomes possible. A possible extension to arbitrary dimension is outlined.
\end{abstract}

\maketitle
% ----------------------------------------------------------------
\section{Introduction and Prelimiaries}
\subsection{Introduction}
A max-stable process (MSP) $\bfxi=(\xi_t)_{t\in K}$ with sample paths in $C(K):=\{g\in \R^K: g \text{ is continuous}\}$ with a compact set $K\subset \R$ has the characteristic property that there are continuous functions $a_n>0, b_n\in C(K)$ for every $n\in\N$ such that
$$
\max_{1\le i\le n} (\bfxi^{(i)} -b_n)/a_n = \left(\max_{1\le i\le n} (\xi_t{(i)} -b_n(t))/a_n(t)\right)_{t\in K} =_d \bfxi,
$$
where $\bfxi^{(i)}$, $i=1,\dotsc,n$, are independent copies of $\bfxi$ and ``$=_d$'' denotes equality in distribution. It is well known (e.g. \citet{dehaf06}), that MSP are the only possible limit processes of linearly standardized maxima of independent and identically distributed processes. This is in complete accordance to the well-established finite dimensional case of extreme value analysis, see \citet{fahure10} and  \citet{dehaf06}, among others.

The theory of continuous max-stable processes is essentially based on the early works of \citet{dehaan84}, \citet{ginhv90} and \citet{dehal01}, followed by recent findings with the focus on different aspects within the theory, e.g. \citet{hulli05,hulli06}, \citet{stota05}, \citet{davmi08}, \citet{kabl09}, \citet{wansto10}, \citet{aulfaho11}.

Moreover, the class of \emph{excursion stable} generalized Pareto processes, which is closely related to the class of max-stable processes, was examined in \citet{buihz08} and \citet{aulfaho11}, the most recent advances on that issue can be found in \citet{ferrdh12}.

There is a very crucial problem of the theory on stochastic processes concerning its relevance in practice: as (continuous) processes as a whole cannot be measured exactly, the question is how to construct those processes (with some characteristic stochastic behavior such as max-stability) from a finite set of observations. As an example one can think of data from a finite set of measuring stations metering the sea level along a coast and one is interested in predicting the sea level between those measuring stations. This means the ``prediction" of a stochastic process in space, not in time.

In the case of max-stable processes (or fields, if the domain has more than one dimension) there are (partial) answers on the arising questions in \citet{wansto11} and \citet{domeyri12} based on conditional sampling. Different to that, our approach is conditionally deterministic.

In this paper we pick up the so-called ``max-linear model'' introduced in \citet{wansto11}:
for arbitrary nonnegative continuous functions  $g_0,\dotsc,g_d$ satisfying condition \eqref{eq:deterministic_functions} below, one obtains an MSP $\bfeta=(\eta_t)_{t\in[0,1]}$ by setting
\begin{equation*}
\eta_s=\max_{j=0,\dotsc,d}\frac{X_j}{g_j(s)},\quad s\in[0,1],
\end{equation*}
where $\bfX=(X_0,\ldots,X_d)$ is a max-stable rv with independent components. The obvious restriction of this model is the required independence of the margins of $\bfX$ which results in the fact that $\bfeta$ always has a discrete spectral measure (the spectral measure is in our setup descripted below represented by the distribution of the generator process $\bfZ$; see \citet{aulfaho11} for details).

In Section \ref{sec:genMLM}, we generalize this model by allowing arbitrary dependence structures of the margins of the max-stable rv $\bm X$. This immediately leads to the main issue of this paper, namely the reconstruction of a max-stable process which is observed only through a finite set of indices: it is shown in Section \ref{sec:predictMSP} that if the random vector is some finite dimensional projection of some initial MSP, the processes resulting from the construction in Section \ref{sec:genMLM} converge uniformly to the original process as the grid of indices gets finer. Moreover, the mean squared error between the  predictive  and the original process is computed at a fixed index, which is useful for practical purposes.

It is also possible to  interpolate  generalized Pareto processes with the same techniques as for MSP which is the content of Section \ref{sec:predictGPP}.

To begin with, the following subsection recalls some basic theory needed in what follows and introduces some notation. For the ease of notation we choose $K=[0,1]$ as domain of the processes, being aware of the fact that all results are valid for an arbitrary compact set $K\subset\R$ as well. The extension of the results to more general domains (in particular higher dimensions of the domain) is not immediately obvious and subject of current research, see Section \ref{sec:arbitrary_dimension} for an outline.

\subsection{Prelimiaries}

We call a random vector (rv) $\bm X=(X_0,\dots,X_d)$ \emph{standard max-stable}, if it is max-stable and each component follows the standard negative exponential distribution, i.e., $P(X_i\le x)=\exp(x)$, $x\le 0$, $i=0,\dots,d$. Different to that, a max-stable rv with unit Fr\'{e}chet margins is commonly called \emph{simple} max-stable in the literature; see, for instance, \citet{dehaf06}. It is well-known (e.g. \citet{dehar77}, \citet{pick81}, \citet[Sections 4.2, 4.3]{fahure10}) that $\bm X$ is standard max-stable if and only if (iff) there exists a rv $\bm Z=(Z_0,\dots,Z_d)$ with $Z_i\in[0,c]$ almost surely for some number $c\ge 1$, and $E(Z_i)=1$, $i=0,\dots,d$, such that
\begin{equation*}
P(\bm X\le \bm x)=\exp(-\norm{\bm x}_D):=\exp\left(-E\left(\max_{0\le i\le d}(\abs{x_i}Z_i)\right)\right),\qquad \bm x\le\bm0\in\R^{d+1}.
\end{equation*}
The condition $Z_i\in[0,c]$ almost surely can be weakened to $P(Z_i\geq0)=1$. Note that $\norm\cdot_D$ defines a norm on $\R^{d+1}$, called \emph{$D$-norm}, with \emph{generator} $\bm Z$. The $D$ means dependence: We have independence of the margins of $\bm X$ iff $\norm\cdot_{D}$ equals the norm $\norm{\bm x}_1=\sum_{i=0}^d\abs{x_i}$, which is generated by $(Z_0,\dots,Z_d)$ being a random permutation of the vector $(d+1,0\dots,0)$. We have complete dependence of the margins of $\bm X$ iff $\norm\cdot_{D}$ is the maximum-norm $\norm{\bm x}_\infty=\max_{0\le i\le d}\abs{x_i}$, which is generated by the constant vector $(Z_0\dots,Z_d)=(1,\dots,1)$. We refer to \citet[Section 4.4] {fahure10} for further details of  $D$-norms.

We call a stochastic process $\bm\eta$ with sample paths in $\bar C^-[0,1]:=\{h\in C[0,1]:h\leq 0\}$ a \emph{standard max-stable process (SMSP)}, if it is a max-stable process with standard negative exponential univariate margins. Again, an MSP with unit Fr\'{e}chet margins is commonly called \emph{simple} MSP in the literature; see, for instance, \citet{dehaf06}.

Denote by $E[0,1]$ the set of those bounded functions $f:[0,1]\to\R$ that have only a finite number of discontinuities, and let $\bar E^-[0,1]$ be the subset of those functions in $E[0,1]$, which attain only non-positive values.

From a mathematical point of view it is quite convenient to introduce the space $E[0,1]$. It allows the incorporation of the finite dimensional marginal distributions of a stochastic process $\bfX$ with sample paths in $\bar C^-[0,1]$ in the form $P(\bfX\le f)$ with a suitable choice of $f\in E[0,1]$. This entails the following characterization of an SMSP in terms of its \emph{distribution function}, which is due to \citet{ginhv90}: A  stochastic process $\bm\eta$ in $C[0,1]$ is an SMSP iff there exists a stochastic process $\bm Z=(Z_t)_{t\in[0,1]}$ with sample paths in $\bar C^+[0,1]:=\{g\in C[0,1]:g\geq 0\}$ with $Z_t\le c$ a.s. for some constant $c\ge 1$, and $E(Z_t)=1$, $t\in[0,1]$, such that
\begin{equation*}
P(\bm\eta\le f)=\exp(-\norm{f}_D):=\exp\left(-E\left(\sup_{t\in[0,1]}(\abs{f(t)}Z_t)\right)\right),\qquad f\in \bar E^-[0,1].
\end{equation*}
A proper choice of the function $f\in \bar E^-[0,1]$ provides the finite dimensional marginal distribution $P(\bm\eta\le f)=P(\eta_{t_i}\le x_i,\,1\le i\le d)$.

The condition $P(\sup_{t\in[0,1]}Z_t\le c )=1$ on the generator process $\bm Z$ can be weakened to $E\left(\sup_{t\in[0,1]}Z_t\right)<\infty$, see \citet[Corollary 9.4.5]{dehaf06}.

Note that $\norm\cdot_D$ defines a norm again, this time on the linear space $E[0,1]$. It is also called $D$-norm with \emph{generator process} $\bm Z$, and we have
\begin{equation*}
\norm f_\infty\le \norm f_D\le \varepsilon_D\norm f_\infty,\qquad f\in E[0,1],
\end{equation*}
where $\norm f_\infty=\sup_{t\in[0,1]}\abs{f(t)}$ and $\varepsilon_D=\norm 1_D=E\left(\norm{\bm Z}_\infty\right)$ is the \emph{extremal coefficient}, cf. \citet{smith90}. This implies $\norm \cdot_D=\norm\cdot_\infty$ iff $\varepsilon_D=1$, cf. \citet{aulfaho11}. Moreover, the preceding inequality shows that each $D$-norm on the space $E[0,1]$ is equivalent to the sup-norm $\norm\cdot_\infty$, which is itself a $D$-norm by putting $Z_t=1$, $t\in[0,1]$.

We conclude this section by introducing generalized Pareto processes as considered in Section \ref{sec:predictGPP}. For the purpose of this paper, the following definition is sufficient: we call a stochastic process $\bm V$ in $\bar C^-[0,1]$ a \emph{standard generalized Pareto process} (SGPP), if there exists a $D$-norm $\norm\cdot_D$ on $E[0,1]$ and some $c>0$, such that $P(\bm V\leq f)=1-\norm f_D$ for all $f\in\bar E^-[0,1]$ with $\norm f_{\infty}\leq c$. For a detailed examination of GPP we refer to \citet{ferrdh12}.

%%%%%%%%%%%%%%%%%%%%%%%%%%%%%%%%%%%%%%%%%%%%%%%%%%%%%%%%%%%%%%%
\section{The generalized max-linear model}\label{sec:genMLM}
%%%%%%%%%%%%%%%%%%%%%%%%%%%%%%%%%%%%%%%%%%%%%%%%%%%%%%%%%%%%%%%%%

Let $\bm X=(X_0,\dotsc,X_d)$ be a standard max-stable rv with pertaining $D$-norm $\norm\cdot_{D_{0,\dotsc,d}}$ on $\R^{d+1}$ generated by $\bm Z=(Z_0,\dotsc,Z_d)$, $d\in\N$, i.\,e.
\begin{equation*}
P\left(\bm X\leq \bm x\right)=\exp\left(-\norm{\bm x}_{D_{0,\dotsc,d}}\right)=\exp\left(-E\left(\max_{i=0,\dotsc,d}\abs{x_i}{Z_i}\right)\right),
\end{equation*}
$\bm x=(x_0,\dotsc,x_d)\leq\bm0.$
Choose arbitrary deterministic functions $g_0,\dotsc,g_d\in\bar C^+[0,1]$ with the property
\begin{equation}\label{eq:deterministic_functions}
\norm{\left(g_0(t),\dotsc,g_d(t)\right)}_{D_{0,\dotsc,d}}=1,\quad t\in[0,1].
\end{equation}
 For instance, in case of independent margins of $\bm X$, we have $\norm\cdot_{D_{0,\dotsc,d}}=\norm\cdot_1$, and condition \eqref{eq:deterministic_functions} becomes
\begin{equation*}
\sum_{i=0}^dg_i(t)=1,\quad t\in[0,1],
\end{equation*}
i.\,e. $g_i(t)$, $i=0,\dotsc,d$, defines a probability distribution on the set $\{0,\dotsc,d\}$ for each $t\in[0,1]$. This is the setup in the max-linear model introduced by \citet{wansto11}. An example for this case is given by the binomial distribution
\begin{equation*}
g_i(t):={d\choose i}t^i(1-t)^{d-i},\quad i=0,\dotsc,d,~t\in[0,1].
\end{equation*}
Put now for $t\in[0,1]$
\begin{equation}\label{eq:generalized_max_linear_model}
\eta_t:=\max_{i=0,\dotsc,d}\frac{X_i}{g_i(t)}.
\end{equation}
The model \eqref{eq:generalized_max_linear_model} is called \emph{generalized max-linear model}. It defines an SMSP as the next lemma shows:

\begin{lemma}\label{lem:max_linear_model_smsp}
The stochastic process $\bm\eta=(\eta_t)_{t\in[0,1]}$ in \eqref{eq:generalized_max_linear_model} defines an SMSP with generator process $\bm{\hat Z}=(\hat Z_t)_{t\in[0,1]}$ given by
\begin{equation*}
\hat Z_t=\max_{i=0,\dotsc,d}\left(g_i(t)Z_i\right),\quad t\in[0,1].
\end{equation*}
\end{lemma}

\begin{proof}
At first we verify that the process $\bm{\hat Z}$ is a generator process indeed. It is obvious that the sample paths of $\bm{\hat Z}$ are in $\bar C^+[0,1]$. Furthermore, we have by construction for each $t\in[0,1]$
\begin{equation*}
E(\hat Z_t)=\norm{\left(g_0(t),\dotsc,g_d(t)\right)}_{D_{0,\dotsc,d}}=1.
\end{equation*}
As  $\norm\cdot_\infty\le \norm\cdot_D$ for an arbitrary $D$-norm, we have $\norm{(g_0(t),\dots,g_d(t))}_\infty\le 1$, $t\in[0,1]$, and, thus, $\hat Z_t\le \max_{i=0,\dots,d}Z_i$, $t\in[0,1]$.

In addition, we have for $f\in\bar E^-[0,1]$
\begin{align*}
P(\bm\eta\leq f) &=P\left(X_i\leq g_i(t)f(t),~i=0,\dotsc,d,~t\in[0,1]\right)\\
%&=P\biggl(X_i\leq \inf_{t\in[0,1]}(g_i(t)f(t)),~i=0,\dotsc,d\biggr)\\
&=P\biggl(X_i\leq-\sup_{t\in[0,1]}(g_i(t)\abs{f(t)}),~i=0,\dotsc,d\biggr)\\
&=\exp\biggl(-\norm{\biggl(\sup_{t\in[0,1]}(g_0(t)\abs{f(t)}),\dotsc,\sup_{t\in[0,1]}(g_d(t)\abs{f(t)})\biggr)}_{D_{0,\dotsc,d}}\biggr)\\
&=\exp\biggl(-E\biggl(\max_{i=0,\dotsc,d}\biggl(\sup_{t\in[0,1]}\biggl(g_i(t)\abs{f(t)}\biggl)Z_i\biggr)\biggr)\biggr)\\
&=\exp\biggl(-E\biggl(\sup_{t\in[0,1]}\biggl(\abs{f(t)}\max_{i=0,\dotsc,d}(g_i(t)Z_i)\biggr)\biggr)\biggr)\\
&=\exp\biggl(-E\biggl(\sup_{t\in[0,1]}\left(\abs{f(t)}\hat Z_t\right)\biggr)\biggr)
\end{align*}
which completes the proof.
\end{proof}

\begin{rem}\label{rem:dropcond_deterministic_functions}\upshape
Condition \eqref{eq:deterministic_functions} ensures that the univariate margins $\eta_t$, $t\in[0,1]$, of the process $\bm\eta$ in model \eqref{eq:generalized_max_linear_model} follow the standard negative exponential distribution $P(\eta_t\leq x)=\exp(x)$, $x\leq 0$. If we drop this condition, we still obtain a max-stable process: Take for $n\in\N$ i.\,i.\,d. copies $\bm\eta^{(1)},\dotsc,\bm\eta^{(n)}$ of $\bm\eta$. We have for $f\in\bar E^-[0,1]$
\begin{align*}
P\,\bigg(n\max_{1\leq k\leq n}\bm\eta^{(k)}&\leq f\bigg)=P\left(X_i\leq \inf_{t\in[0,1]}\left(\frac{g_i(t)f(t)}{n}\right),~i=0,\dotsc,d\right)^n\\
&=\exp\biggl(-\norm{\biggl(\sup_{t\in[0,1]}\left(g_0(t)\abs{f(t)}\right),\dotsc,\sup_{t\in[0,1]}\biggl(g_d(t)\abs{f(t)}\biggr)\biggr)}_{D_{0,\dotsc,d}}\biggr)\\
&=P(\bm\eta\leq f).
\end{align*}
The univariate margins of $\bm\eta$ are now given by
\begin{equation}\label{eq:margins_dropcond_deterministic_functions}
P(\eta_t\leq x)=\exp\left(\norm{\left(g_0(t),\dotsc,g_d(t)\right)}_{D_{0,\dotsc,d}}\cdot x\right),\quad x\leq 0,~t\in[0,1].
\end{equation}
Note that the above calculations also give an alternative proof of Lemma \ref{lem:max_linear_model_smsp}, except we do not obtain the generator process of $\bm\eta$ with this approach.
\end{rem}

In model \eqref{eq:generalized_max_linear_model} we have not made any further assumptions on the $D$-norm $\norm\cdot_{D_{0,\dotsc,d}}$, that is, on the dependence structure of the random variables $X_0,\dotsc,X_d$. The special case $\norm\cdot_{D_{0,\dotsc,d}}=\norm\cdot_1$ characterizes the independence of $X_0,\dotsc,X_d$. This is the regular \emph{max-linear model}, cf. \citet{wansto11}.

On the contrary, $\norm\cdot_{D_{0,\dotsc,d}}=\norm\cdot_{\infty}$ provides the case of complete dependence $X_0=\cdots=X_d$ a.\,s. with the constant generator $Z_0=\cdots=Z_d=1$. Thus, condition \eqref{eq:deterministic_functions} becomes $\max_{i=0,\dotsc,d}g_i(t)=1$, $t\in[0,1]$, and therefore
\begin{equation*}
\hat Z_t=\max_{i=0,\dotsc,d}\left(g_i(t)Z_i\right)=\max_{i=0,\dotsc,d}g_i(t)=1,\quad t\in[0,1].
\end{equation*}

%%%%%%%%%%%%%%%%%%%%%%%%%%%%%%%%%%%%%%%%%%%%%
\section{Reconstruction of SMSP}\label{sec:predictMSP}
%%%%%%%%%%%%%%%%%%%%%%%%%%%%%%%%%%%%%%%%%%%%%%

The preceding approach offers a way to reconstruct an SMSP in an appropriate way. Let $\bfeta=(\eta_t)_{t\in[0,1]}$ be an SMSP with generator process $\bfZ=(Z_t)_{t\in[0,1]}$ and $D$-norm $\norm\cdot_D$. Choose a grid $0=: s_0<s_1<\dots<s_{d-1}<s_d:=1$ of points within $[0,1]$. Then $\left(\eta_{s_0},\dots,\eta_{s_d}\right)$ is a standard max-stable rv in $\R^{d+1}$ with pertaining $D$-norm $\norm\cdot_{D_{0,\dotsc,d}}$ generated by $\left(Z_{s_0},\dots,Z_{s_d}\right)$.

The aim of this section is to define some SMSP $\hat{\bm\eta}=(\hat\eta_t)_{t\in[0,1]}$ for which $\hat\eta_{s_i}=\eta_{s_i},\ i=0,\ldots,d$, holds, i.e. $\hat{\bm\eta}$ \emph{interpolates} the finite dimensional projections $\left(\eta_{s_0},\dots,\eta_{s_d}\right)$ of the original SMSP $\bfeta$ in an appropriate way. This will be done by means of a special case of the generalized max-linear model, i.e., by a particular choice of the functions $g_i$ in equation \eqref{eq:generalized_max_linear_model}. We show that this way of predicting the original MSP $\bfeta$ in space is reasonable, as the pointwise mean squared error $\MSE\left(\hat\eta_t^{(d)}\right):=E\left(\left(\eta_t-\hat\eta_t^{(d)}\right)^2\right)$ diminishes for all $t\in[0,1]$ as $d$ increases. Moreover, we establish uniform convergence of the ``predictive'' processes and the corresponding generator processes to the original ones.

\subsection{Uniform convergence of the discretized versions}

As we have shown in Lemma \ref{lem:max_linear_model_smsp}, the stochastic process $\hat{\bm\eta}=(\hat\eta_t)_{t\in[0,1]}$,
\begin{equation*}
\hat\eta_t=\max_{i=0,\dotsc,d}\frac{\eta_{s_i}}{g_i(t)},\quad t\in[0,1],
\end{equation*}
defines an SMSP with generator process $\hat{\bm Z}=(\hat Z_t)_{t\in[0,1]}$, given by
\begin{equation*}
\hat Z_t=\max_{i=0,\dotsc,d}\left(g_i(t)Z_{s_i}\right),\quad t\in[0,1],
\end{equation*}
for arbitrary  functions $g_0,\dots,g_d$ in $\bar C^+[0,1]$ that satisfy condition \eqref{eq:deterministic_functions}. We are going to specialize them now.

Denote by $\norm\cdot_{D_{i-1,i}}$ the $D$-norm pertaining to the bivariate rv $(\eta_{s_{i-1}},\eta_{s_i})$, $i=1,\dotsc,d$. Put
\begin{align*}
g_0^{\ast}(t)&:=\begin{cases}\dfrac{s_{1}-t}{\norm{(s_{1}-t,t)}_{D_{0,1}}},\quad &t\in[0,s_1], \\ 0,\quad &\text{else},\end{cases}\\
g_i^\ast(t)&:=\begin{cases}\dfrac{t-s_{i-1}}{\norm{(s_i-t,t-s_{i-1})}_{D_{i-1,i}}},\quad &t\in[s_{i-1},s_i], \\ \dfrac{s_{i+1}-t}{\norm{(s_{i+1}-t,t-s_{i})}_{D_{i,i+1}}},\quad &t\in[s_i,s_{i+1}], \\ 0,\quad &\text{else},\end{cases}\quad i=1,\dotsc,d-1,\\
g_d^\ast(t)&:=\begin{cases}\dfrac{t-s_{d-1}}{\norm{(s_d-t,t-s_{d-1})}_{D_{d-1,d}}},\quad &t\in[s_{d-1},s_d], \\ 0,\quad &\text{else}.\end{cases}
\end{align*}
Clearly,  $g_0^\ast,\dotsc,g_d^\ast\in\bar C^+[0,1]$ since the fact that a $D$-norm is standardized implies
\begin{equation*}
\lim_{t\uparrow s_i}g_i^\ast(t)=\frac{s_i-s_{i-1}}{\norm{(0,s_i-s_{i-1})}_{D_{i-1,i}}}=1=\frac{s_{i+1}-s_{i}}{\norm{(s_{i+1}-s_{i},0)}_{D_{i-1,i}}}=\lim_{t\downarrow s_i}g_i^\ast(t).
\end{equation*}
Moreover, we have for $t\in[s_{i-1},s_i]$, $i=1,\dotsc,d$,
\begin{equation*}
\norm{\left(g_0^\ast(t),\dotsc,g_d^\ast(t)\right)}_{D_{0,\dotsc,d}}=\norm{\left(g^\ast_{i-1}(t),g^\ast_i(t)\right)}_{D_{i-1,i}}=1.
\end{equation*}
Hence, the functions $g_0^\ast,\dotsc,g_d^\ast$ are suitable for the generalized max-linear model \eqref{eq:generalized_max_linear_model}. In addition, they have the following property:

\begin{lemma}\label{lem:structure_deterministic_functions}
The functions $g_0^\ast,\dotsc,g_d^\ast$ defined above satisfy
\begin{equation*}
\norm{g_i^\ast}_{\infty}=g^\ast_i(s_i)=1,\quad i=0,\dotsc,d.
\end{equation*}
\end{lemma}

In view of their properties described above, the functions $g^\ast_i$ can be viewed as kernel functions quite similar to kernels in nonparametric kernel density estimators. Each function $g^\ast_i(t)$ has maximum value 1 at $t=s_i$ and, with the distance between $t$ and $s_i$ increasing, the value $g^\ast_i(t)$ shrinks to zero. This view provides also the idea behind the extension of this approach to higher dimension as outlined in Section \ref{sec:arbitrary_dimension}.

\begin{proof}[Proof of Lemma \ref{lem:structure_deterministic_functions}]
From the fact that a $D$-norm is monotone and standardized we obtain for $i=1,\dotsc,d-1$ and $t\in[s_{i-1},s_i)$
\begin{equation*}
g^\ast_i(t)=\frac{t-s_{i-1}}{{\norm{\left(s_i-t,t-s_{i-1}\right)}_{D_{i-1,i}}}}=\frac{1}{\norm{\left(\frac{s_i-t}{t-s_{i-1}},1\right)}_{D_{i-1,i}}}\leq\frac{1}{\norm{(0,1)}_{D_{i-1,i}}}=1,
\end{equation*}
and for $t\in[s_i,s_{i+1})$
\begin{equation*}
g^\ast_{i}(t)=\frac{s_{i+1}-t}{{\norm{\left(s_{i+1}-t,t-s_{i}\right)}_{D_{i,i+1}}}}=\frac{1}{\norm{\left(1,\frac{t-s_{i}}{s_{i+1}-t}\right)}_{D_{i,i+1}}}\leq\frac{1}{\norm{(1,0)}_{D_{i,i+1}}}=1.
\end{equation*}
Analogously, we have $g^\ast_0\leq 1$ and $g^\ast_d\leq 1$. The assertion now follows since $g^\ast_i(s_i)=1$, $i=0,\dotsc,d$.
\end{proof}

The SMSP $\hat{\bm\eta}=(\hat\eta_t)_{t\in[0,1]}$ that is generated by the generalized max-linear model with these particular functions $g^\ast_0,\dotsc,g^\ast_d$ is given by
\begin{align}
\hat\eta_t&=\max\left(\frac{\eta_{s_{i-1}}}{g^\ast_{i-1}(t)},\frac{\eta_{s_i}}{g^\ast_i(t)}\right)\label{eq:discretizing_version_smsp}\\
&=\norm{(s_i-t,t-s_{i-1})}_{D_{i-1,i}}\max\left(\frac{\eta_{s_{i-1}}}{s_i-t},\frac{\eta_{s_i}}{t-s_{i-1}}\right),\quad t\in[s_{i-1},s_i],~i=1,\dotsc,d.\nonumber
\end{align}
Note that $\eta_{s_i}< 0$ almost surely, $i=0,\dotsc,d$. This implies that the maximum taken over $d+1$ points in \eqref{eq:generalized_max_linear_model} goes down to a maximum taken over only two points in \eqref{eq:discretizing_version_smsp} since all except two of the $g_i$ vanish in $t\in[s_{i-1},s_i]$, $i=1,\dotsc,d$. We have, moreover,
\[
\hat\eta_{s_i}=\eta_{s_i},\qquad i=0,\dots,d,
\]
so the above process interpolates the rv $\left(\eta_{s_0},\dots,\eta_{s_d}\right)$.

To give an example, put $\norm{(x_1,x_2)}_{D_{i-1,i}}:=\norm{(x_1,x_2)}_\lambda:=(\abs{x_1}^\lambda+ \abs{x_2}^\lambda)^{1/\lambda}$, $1\le\lambda\le \infty$ for every $1\leq i\leq d$, i.\,e. each bivariate $D$-norm $\norm\cdot_{D_{i-1,i}}$ is the logistic one. For $\lambda<\infty$, the logistic $D$-norm is generated by $Z_i:=X_i/\Gamma(1-\lambda^{-1})$, $i=1,2$, where $X_1,X_2$ are independent and identically unit Fr\'{e}chet distributed rv and $\Gamma(\cdot)$ denotes the gamma function. In this case, we obtain the representation
\begin{align}
\hat \eta_t =((s_i-t)^\lambda)+(t-s_{i-1})^\lambda)^{1/\lambda} &\max\left(\frac{\eta_{s_{i-1}}}{s_i-t},\frac{\eta_{s_i}}{t-s_{i-1}}\right),\label{eq:logistic_discretized_version}\\
&t\in[s_{i-1},s_i],~i=1,\dotsc,d.\nonumber
\end{align}

In summary, we have proven the following result.

\begin{cor}\label{cor:generation_of_smsp_discretized}
Let $\bm\eta=(\eta_t)_{t\in[0,1]}$ be an SMSP with generator $\bm Z=(Z_t)_{t\in[0,1]}$, and let $0:=s_0<s_1<,\dotsc,<s_{d-1}<s_d:=1$ be a grid of points in the interval $[0,1]$. The process $\hat{\bfeta}=(\hat\eta_t)_{t\in[0,1]}$ defined in \eqref{eq:discretizing_version_smsp} is an SMSP with generator process $\hat{\bm Z}=(\hat Z_t)_{t\in[0,1]}$, where
\begin{equation}\label{eq:discretizing_version_generator}
\hat Z_t=\frac{\max\Big((s_i-t)Z_{s_{i-1}},(t-s_{i-1})Z_{s_i}\Big)}{\norm{(s_i-t,t-s_{i-1})}_{D_{i-1,i}}},\quad t\in[s_{i-1},s_i],~i=1,\dotsc,d.
\end{equation}
The processes $\hat{\bm\eta}$ and $\hat{\bm Z}$ interpolate the rv $(\eta_{s_0},\dotsc,\eta_{s_d})$ and $(Z_{s_0},\dotsc,Z_{s_d})$, respectively.
\end{cor}

We call  $\bm{\hat\eta}$ the \emph{discretized version} of $\bfeta$ and $\bm{\hat Z}$ the discretized version of $\bfZ$, both with grid $\set{s_0,\dots,s_d}$. Next we show that the preceding approach allows the approximation of an underlying SMSP based on multivariate observations; that is, the discretized version of the underlying SMSP converges to this very process in a strong sense. We need the following two lemmata which provide some technical insight in the strucure of the chosen max-linear model.

\begin{lemma}\label{lem:max_and_min_discretized}
The SMSP defined in \eqref{eq:discretizing_version_smsp} fulfills for $i=1,\dots,d$
\[
\sup_{t\in[s_{i-1},s_i]} \hat\eta_t=\max\left(\eta_{s_{i-1}},\eta_{s_i}\right),
\]
and
\begin{equation*}
\inf_{t\in[s_{i-1},s_i]} \hat\eta_t=-\norm{(\eta_{s_{i-1}},\eta_{s_i})}_{D_{s_{i-1},s_{i}}}.
\end{equation*}
This minimum is attained for $t=(s_{i-1}\eta_{s_{i-1}}+s_i\eta_{s_i})/(\eta_{s_{i-1}}+\eta_{s_i})$.
\end{lemma}

\begin{proof}
We know from Lemma \ref{lem:structure_deterministic_functions} that $g^\ast_{i-1}(t),g^\ast_i(t)\leq 1$ for an arbitrary $i=1,\dotsc,d$ and $t\in[s_{i-1},s_i]$. Hence,
\begin{equation*}
\hat\eta_t=\max\left(\frac{\eta_{s_{i-1}}}{g^\ast_{i-1}(t)},\frac{\eta_{s_i}}{g^\ast_i(t)}\right)\leq\max(\eta_{s_{i-1}},\eta_{s_i})
\end{equation*}
for $i=1,\dotsc,d$ and $t\in[s_{i-1},s_i]$, which yields the first part of the assertion. Recall that $\eta_{s_i}< 0$ with probability one, $i=0\dotsc,d.$

Moreover, we have for $t\in(s_{i-1},s_i)$
\begin{equation*}
\frac{\eta_{s_{i-1}}}{s_i-t}\le \frac{\eta_{s_i}}{t-s_{i-1}}\iff \frac{s_i-t}{t-s_{i-1}}\leq\frac{\eta_{s_{i-1}}}{\eta_{s_i}}\iff t\ge \frac{s_{i-1}\eta_{s_{i-1}}+s_i\eta_{s_i}}{\eta_{s_{i-1}}+\eta_{s_i}},
\end{equation*}
where equality in one of these expressions occurs iff it does in the other two. In this case of equality we have
\begin{equation*}
\hat\eta_t=\norm{\left(s_i-t,t-s_{i-1}\right)}_{D_{i-1,i}}\cdot\frac{\eta_{s_i}}{t-s_{i-1}}=-\norm{(\eta_{s_{i-1}},\eta_{s_i})}_{D_{i-1,i}}.
\end{equation*}
On the other hand, the monotonicity of a $D$-norm implies for every $t\in(s_{i-1},s_i)$ with $t\ge (s_{i-1}\eta_{s_{i-1}}+s_i\eta_{s_i})/(\eta_{s_{i-1}}+\eta_{s_i})$
\begin{align*}
\hat\eta_t&\geq\norm{\left(s_i-t,t-s_{i-1}\right)}_{D_{i-1,i}} \frac{\eta_{s_i}}{t-s_{i-1}}\\
&={\norm{\left(\frac{s_i-t}{t-s_{i-1}},1\right)}_{D_{i-1,i}}}  \eta_{s_i}\\
&\geq \norm{\left(\frac{\eta_{s_{i-1}}}{\eta_{s_i}},1\right)}_{D_{i-1,i}}  \eta_{s_i}\\
&=-\norm{(\eta_{s_{i-1}},\eta_{s_i})}_{D_{i-1,i}}.
\end{align*}
Recall again that $\eta_{s_i}<0$ almost surely. The case $t\le (s_{i-1}\eta_{s_{i-1}}+s_i\eta_{s_i})/(\eta_{s_{i-1}}+\eta_{s_i})$ works analogously.
\end{proof}

As an immediate consequence of the preceding result we obtain for $x\le 0$
\begin{equation*}
\bm{\hat\eta}\le x\iff \max\left(\eta_{s_0},\dots,\eta_{s_d}\right)\le x
\end{equation*}
and
\begin{equation*}
\bm{\hat\eta} > x\iff \max_{1\le i\le d}\norm{\left(\eta_{s_{i-1}},\eta_{s_i}\right)}_{D_{i-1,i}} < -x.
\end{equation*}

In order to visualize the interpolation scheme of this particular generalized max-linear model, we plot some discretized versions with different grids and bivariate $D$-norms $\norm\cdot_{D_{i-1,i}}$. For the sake of simplicity, the underlying path $\bm\eta(\omega)$ in this example shall not arise from a simulation of an actual SMSP, but rather is replaced by a smooth deterministic continous function on $[0,1]$.

More precisely, we choose in the following picture $\eta_t(\omega):=7.5(0.16t-0.5t^2+t^3/3)-0.125$, $t\in[0,1]$ which is represented by the dashed curve. The solid line in each plot is the discretized version $\hat{\bm\eta}(\omega)$ of this path. We use equidistant grids of dimension $d=5$, $d=10$ and $d=20$. Each bivariate $D$-norm $\norm\cdot_{D_{i-1,i}}$ are logistic norms such that the discretized versions are given by formula \eqref{eq:logistic_discretized_version} with $\lambda=2$, $\lambda=4$ and $\lambda=8$.

The plots apparently show that the approximation of the original process through a discrtized version improves as the dimension $d$ gets higher and as the bivariate $D$-norms get closer to complete dependence case.

\clearpage

\begin{figure}
\includegraphics[width=0.8\textwidth]{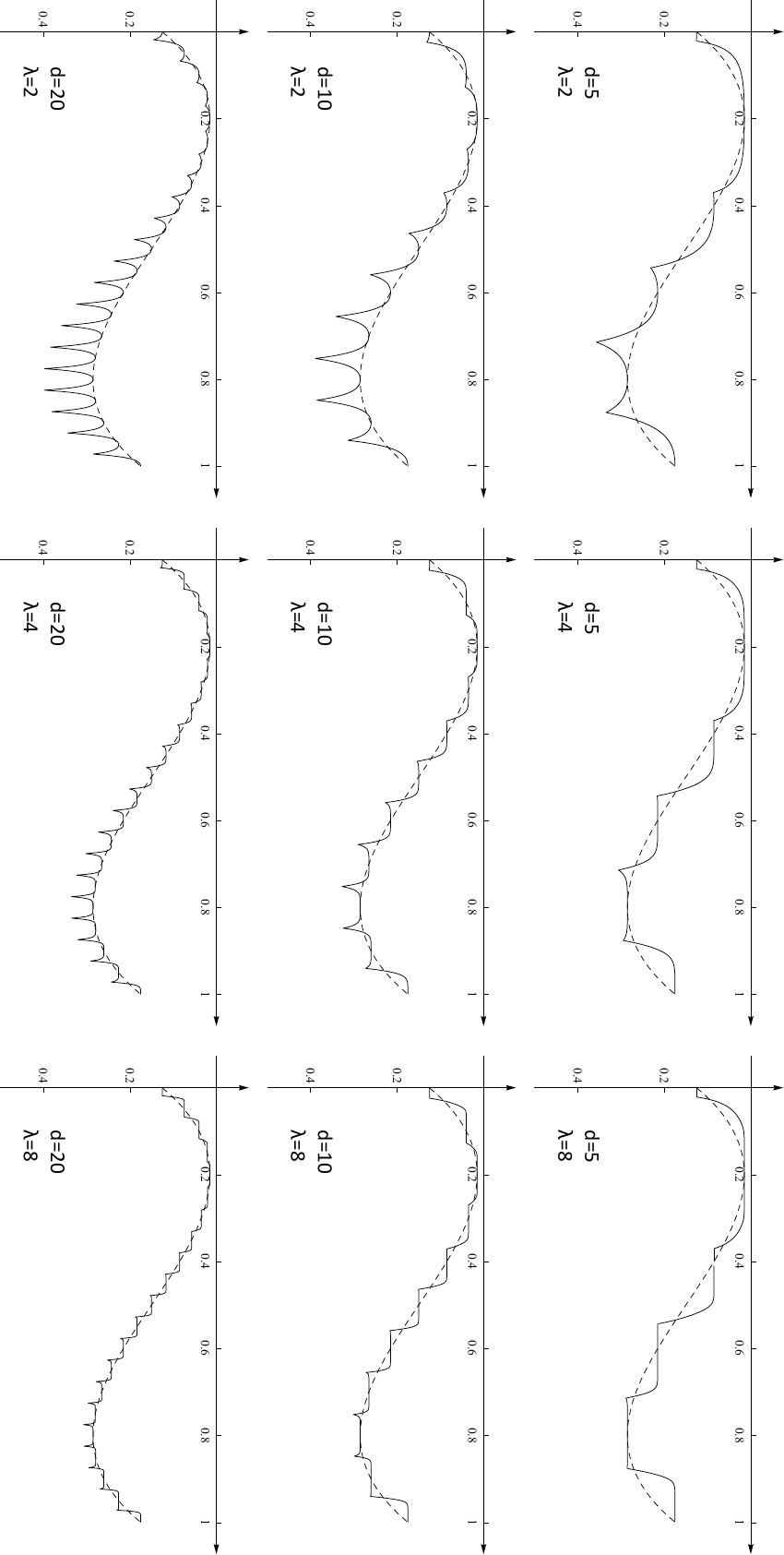}\bigskip\bigskip

\caption{Plots of logistic type discretized versions (solid) of a deterministic function that stands for a path of an SMSP (dashed). }
\end{figure}

\clearpage

The next lemma is on the structure of the underlying generator processes. It can easily be proven by arguments that are very similar to those of the proof of Lemma \ref{lem:structure_deterministic_functions}.

\begin{lemma}\label{lem:max_and_min_discretized_generator}
The generator process defined in \eqref{eq:discretizing_version_generator} fulfills for $i=1,\dots,d$
\begin{equation*}
\sup_{t\in[s_{i-1},s_i]} \hat Z_t=\max\left(Z_{s_{i-1}},Z_{s_i}\right).
\end{equation*}
In particular, the extremal coefficient $E\left(\norm{\hat{\bm Z}}_{\infty}\right)$ of the SMSP $\hat{\bm\eta}$ coincides with the extremal coefficient $E(\max_{i=0,\dotsc,d}Z_{s_i})$ of the rv $(\eta_{s_0},\dotsc,\eta_{s_d})$. Moreover, for $i=1,\dotsc,d$,
\begin{equation*}
\inf_{t\in[s_{i-1},s_i]} \hat Z_t=\begin{cases}\left(\norm{(1/Z_{s_{i-1}},1/Z_{s_i})}_{D_{i-1,i}}\right)^{-1} &\text{if}\quad Z_{s_{i-1}},~Z_{s_i}>0,\\ 0 &\text{else.}\end{cases}
\end{equation*}
In the first case, the minimum is attained for $t=(s_{i-1}Z_{s_i}+s_iZ_{s_{i-1}})/\left(Z_{s_{i-1}}+Z_{s_i}\right)$.
\end{lemma}

%\begin{proof}
%We know from Lemma \ref{lem:structure_deterministic_functions} that $g^\ast_{i-1}(t),g^\ast_i(t)\leq 1$ for an arbitrary $i=1,\dotsc,d$ and $t\in[s_{i-1},s_i]$. Hence,
%\begin{equation*}
%\hat Z_t=\max\Big(g^\ast_{i-1}(t)Z_{s_{i-1}},g^\ast_i(t)Z_{s_i}\Big)\leq\max(Z_{s_{i-1}},Z_{s_i})
%\end{equation*}
%for $i=1,\dotsc,d$ and $t\in[s_{i-1},s_i]$, which yields the first part of the assertion.
%
%Moreover, we have for $t\in(s_{i-1},s_i)$ in case of $Z_{s_{i-1}},Z_{s_i}>0$
%\begin{equation*}
%(s_i-t)Z_{s_{i-1}}\le (t-s_{i-1})Z_{s_i}\iff \frac{s_i-t}{t-s_{i-1}}\leq\frac{Z_{s_i}}{Z_{s_{i-1}}}\iff t\ge \frac{s_{i-1}Z_{s_i}+s_iZ_{s_{i-1}}}{Z_{s_{i-1}}+Z_{s_i}},
%\end{equation*}
%where equality in one of these expressions occurs iff it does in the other two. In this case of equality we have
%\begin{equation*}
%\hat Z_t=\frac{(t-s_{i-1})Z_{s_i}}{\norm{\left(s_i-t,t-s_{i-1}\right)}_{D_{i-1,i}}}=\frac{1}{\norm{(1/Z_{s_{i-1}},1/Z_{s_i})}_{D_{i-1,i}}}.
%\end{equation*}
%On the other hand, the monotonicity of a $D$-norm implies for every $t\in(s_{i-1},s_i)$ with $t\ge (s_{i-1}Z_{s_{i}}+s_iZ_{s_{i-1}})/(Z_{s_{i-1}}+Z_{s_i})$
%\begin{align*}
%\hat Z_t&\geq\frac{(t-s_{i-1})Z_{s_i}}{\norm{\left(s_i-t,t-s_{i-1}\right)}_{D_{i-1,i}}}\\
%&=\left({\norm{\left(\frac{s_i-t}{t-s_{i-1}},1\right)}_{D_{i-1,i}}}\right)^{-1} Z_{s_i}\\
%&\geq \left(\norm{\left(\frac{Z_{s_i}}{Z_{s_{i-1}}},1\right)}_{D_{i-1,i}}\right)^{-1} Z_{s_i}\\
%&=\frac{1}{\norm{(1/Z_{s_{i-1}},1/Z_{s_i})}_{D_{i-1,i}}}.
%\end{align*}
%The case $t\le (s_{i-1}Z_{s_i}+s_iZ_{s_{i-1}})/(Z_{s_{i-1}}+Z_{s_i})$ is shown analogously.
%\end{proof}

So far we have only considered a \emph{fixed} discretized version of an SMSP. The next step is to examine a \emph{sequence} of discretized versions with certain grids whose diameter converges to zero. It turns out that such a sequence converges to the initial SMSP in the function space $C[0,1]$ equipped with the sup-norm. Thus, our method is suitable to reconstruct the initial process.

Let
\begin{equation*}
\mathcal G_d:=\{s_0^{(d)},s_1^{(d)}\dotsc,s_d^{(d)}\}, \quad 0=:s_0^{(d)}<s_1^{(d)}<\cdots<s_d^{(d)}:=1,\quad d\in\N,
\end{equation*}
be a sequence of grids in $[0,1]$ with diameter
\begin{equation*}
\kappa_d:=\max_{i=1,\dotsc,d}\left(s_i^{(d)}-s_{i-1}^{(d)}\right)\to_{d\to\infty}0.
\end{equation*}
Let $\bm{\hat\eta}^{(d)}=(\hat \eta^{(d)}_t)_{t\in[0,1]}$ be the discretized version of an SMSP $\bfeta=(\eta_t)_{t\in[0,1]}$ with grid $\mathcal G_d$. Denote by $\bm{\hat Z}^{(d)}=(\hat Z^{(d)}_t)_{t\in[0,1]}$ and $\bfZ=(Z_t)_{t\in[0,1]}$ the generator processes pertaining to $\bm{\hat\eta}^{(d)}$ and $\bfeta$, respectively.

\begin{theorem}\label{the:convergence_of_discretized}
The processes $\bm{\hat\eta}^{(d)}$ and $\bm{\hat Z}^{(d)}$, $d\in\N$, converge uniformly to $\bfeta$ and $\bfZ$ pathwise, i.\,e. $\norm{\bm{\hat\eta}^{(d)}-\bfeta}_{\infty}\to_{d\to\infty}0$ and $\norm{\bm{\hat Z}^{(d)}-\bfZ}_{\infty}\to_{d\to\infty}0$ with probability one.
\end{theorem}

\begin{proof}
Denote by $[t]_d$, $d\in\N$, the left neighbor of $t\in[0,1]$ among $\mathcal G_d$, and by $\langle t\rangle_d$, $d\in\N$, the right neighbor of $t\in[0,1]$ among $\mathcal G_d$. Choose a sequence $s^{(d)}\in[0,1]$, $d\in\N$, with $s^{(d)}\to_{d\in\N}s\in[0,1]$. Then obviously $[s^{(d)}]_d\to_{d\to\infty}s$ and $\langle s^{(d)}\rangle_d\to_{d\to\infty}s$. Hence we obtain by Lemma \ref{lem:max_and_min_discretized}, and the continuity of the process $\bfeta$
\begin{equation*}
\hat \eta^{(d)}_{s^{(d)}}\leq\max_{s\in\left[[s^{(d)}]_d,\langle s^{(d)}\rangle_d\right]}\hat\eta_s^{(d)}=\max\left(\eta_{[s^{(d)}]_d},\eta_{\langle s^{(d)}\rangle_d}\right)\to_{d\to\infty}\eta_s,
\end{equation*}
as well as
\begin{equation*}
\hat \eta^{(d)}_{s^{(d)}}\geq\min_{s\in\left[[s^{(d)}]_d,\langle s^{(d)}\rangle_d\right]}\hat\eta_s^{(d)}=-\norm{\left(\eta_{[s^{(d)}]_d},\eta_{\langle s^{(d)}\rangle_d}\right)}_{D_{[s^{(d)}]_d,\langle s^{(d)}\rangle_d}}\to_{d\to\infty}\eta_s,
\end{equation*}
where $\norm{\cdot}_{D_{[s^{(d)}]_d,\langle s^{(d)}\rangle_d}}$ denotes the $D$-norm pertaining to $\left(\eta_{[s^{(d)}]_d},\eta_{\langle s^{(d)}\rangle_d}\right)$. Hence the first part of the assertion is proven.

Now we show that $\bm{\hat Z}^{(d)}\to_{d\to\infty}\bfZ$ in $(C[0,1],\norm\cdot_{\infty})$. If $Z_s\neq0$, the continuity of $\bfZ$ implies $Z_{[s^{(d)}]_d}\neq0\neq Z_{\langle s^{(d)}\rangle_d}$ for sufficiently large values of $d$. Repeating the above arguments, the assertion now follows by Lemma \ref{lem:max_and_min_discretized_generator}. If $Z_s=0$, the continuity of $\bfZ$ implies
\begin{equation*}
\hat Z^{(d)}_{s^{(d)}}\leq 2\max\left(Z_{[s^{(d)}]_d},Z_{\langle s^{(d)}\rangle_d}\right)\to_{d\to\infty}2Z_s=0,
\end{equation*}
which completes the proof. Check that $\norm{(\langle s^{(d)}\rangle_d-t,t-[s^{(d)}]_d)}_D\geq 1/2$
since every $D$-norm is monotone and standardized.
\end{proof}

The preceding theorem is the main reason why we consider the discretized version $\hat{\bm\eta}$ of an SMSP $\bm\eta$ a reasonable predictor of this process, where the prediction is done in space, not in time. The predictions $\hat\eta_t$ of the points $\eta_t$, $t\in[0,1]$, only depend on the multivariate observations $(\eta_{s_0},\dotsc,\eta_{s_d})$. More precisely, the only additional thing we need to know to make these predictions is the set of the adjacent bivariate marginal distributions of $(\eta_{s_0},\dotsc,\eta_{s_d})$, that is, the bivariate $D$-norms $\norm\cdot_{D_{i-1,i}}$, $i=1,\dotsc,d$. This might, however, be a restrictive condition in praxis and suggests the problem to fit models of bivariate $D$-norms to data, which is, however, beyond the scope of the present paper and requires future investigation.

The following results, however, are obvious. Let $\hat\eta_t$ be a point of the discretized version defined in \eqref{eq:discretizing_version_smsp} and define a \emph{defective discretized version} via
\begin{align*}
\tilde\eta_t:=\norm{(s_i-t,t-s_{i-1})}_{\tilde D_i}\max\left(\frac{\eta_{s_{i-1}}}{s_i-t},\frac{\eta_{s_i}}{t-s_{i-1}}\right),\quad t\in[s_{i-1},s_i],~i=1,\dotsc,d,
\end{align*}
where $\norm\cdot_{\tilde D_i}$ is an arbitrary norm on $\R^2$ which we call the \emph{defective norm}. Then for every $t\in[s_{i-1},s_i]$, $i=1,\dotsc,d$,
\begin{align*}
&\abs{\hat\eta_t-\tilde\eta_t}=\\
&\quad\abs{\norm{(s_i-t,t-s_{i-1})}_{D_{i-1,i}}-\norm{(s_i-t,t-s_{i-1})}_{\tilde D_i}}\min\left(\frac{-\eta_{s_{i-1}}}{s_i-t},\frac{-\eta_{s_{i}}}{t-s_{i-1}}\right).
\end{align*}
In particular, we have $\tilde\eta_{s_i}=\hat\eta_{s_i}=\eta_{s_i}$, $i=0,\dotsc,d$. This means that we obtain an interpolating process even if we replace the $D$-norm $\norm\cdot_{D_{i-1,i}}$ by the defective norm $\norm\cdot_{\tilde D_i}$. Futhermore, the defective discretized version still defines an MSP with sample paths in $\bar C^-[0,1]$. Its univariate marginal distributions are given by
\begin{equation*}
P(\tilde\eta_t\leq x)=\exp\left(\frac{\norm{(s_i-t,t-s_{i-1})}_{D_{i-1,i}}}{\norm{(s_i-t,t-s_{i-1})}_{\tilde D_{i}}}x\right),\quad x\leq 0,~t\in[s_{i-1},s_i],~i=1,\dotsc,d.
\end{equation*}
In addition to this, the assertions in Lemma \ref{lem:max_and_min_discretized} also hold for the defective discretized version in case we know that each defective norm $\norm\cdot_{\tilde D_i}$ is monotone and standardized. Repeating the arguments in the proof of Theorem \ref{the:convergence_of_discretized} now shows that the uniform convergence towards the original process $\bm\eta$ is retained if we replace the norms $\norm\cdot_{D_{i-1,i}}$ by arbitrary monotone and standardized norms $\norm\cdot_{\tilde D_i}$. In that case, the only property of the discretized version that we have to drop is the standardization of the univariate margins.

\subsection{The Mean Squared Error of the Discretized Version}

Considering $\hat\eta_t$ a predictor of $\eta_t$ we can ask for further properties such as the \emph{mean squared error}, which is our next aim. For that purpose, we formulate the following lemma. It applies to bivariate standard max-stable rv in general.

\begin{lemma}\label{lem:covariance_bivariate_smrv}
Let $(X,Y)$ be a bivariate standard max-stable rv, i.\,e. there exists some $D$-norm $\norm\cdot_D$ such that $P(X\leq x,Y\leq y)=\exp\left(-\norm{(x,y)}_D\right)$, $x,y\leq 0$. Then
\begin{equation*}
E(XY)=\int_{0}^{\infty}\frac{1}{\norm{(1,t)}^2_D}~dt.
\end{equation*}
In particular, the covariance and the correlation coefficient $\varrho$ of $X$ and $Y$ are given by
\begin{equation*}
\Cov(X,Y)=\int_{0}^{\infty}\frac{1}{\norm{(1,t)}^2_D}~dt-1=\varrho(X,Y).
\end{equation*}
\end{lemma}

\begin{proof}
%Recall that the expected value of a negative exponentially distributed random variable $\xi$ with parameter $\lambda>0$ is given by $E(\xi)=-1/\lambda.$
%\begin{equation*}
%E(\xi)=\int_{-\infty}^0x\lambda\exp(\lambda x)~dx=-1/\lambda.
%\end{equation*}
Elementary calculations show
\begin{align*}
E(XY)&=\int_{-\infty}^0\int_{-\infty}^0P(X\leq x,Y\leq y)~dx~dy\\
&=\int_{-\infty}^0\int_{-\infty}^0\exp\left(-\norm{(x,y)}_D\right)~dx~dy\\
%&=\int_{-\infty}^0\int_{-\infty}^0\exp\left(x\norm{(1,y/x)}_D\right)~dx~dy\\
&=-\int_{0}^{\infty}\int_{-\infty}^{0}x\exp\left(x\norm{(1,u)}_D\right)~dx~du\\
%&=-\int_0^{\infty}\frac{1}{\norm{(1,u)}_D}\int_{-\infty}^0x\norm{(1,u)}_D\exp\left(x\norm{(1,u)}_D\right)~dx~du\\
&=\int_{0}^{\infty}\frac{1}{\norm{(1,u)}^2_D}~du.
\end{align*}
The remaining assertions follow from the fact that $E(X)=E(Y)=-1$ and $\Var(X)=\Var(Y)=1$.
\end{proof}

\begin{exam}\label{exam:covariances}\upshape
In accordance to the characterization of the independence and complete dependence case in terms of $D$-norms, we obtain in case of $\norm\cdot_D=\norm\cdot_1$
\begin{equation*}
\Cov(X,Y)=\int_0^{\infty}\frac{1}{(u+1)^2}~du-1=0
\end{equation*}
and in case of $\norm\cdot_D=\norm\cdot_{\infty}$
\begin{equation*}
\Cov(X,Y)=\int_0^{\infty}\frac{1}{\left(\max(u,1)\right)^2}~du-1=1.
\end{equation*}
In particular, we have $\Cov(X,Y)=\varrho(X,Y)\in[0,1]$ for every bivariate standard max-stable rv $(X,Y)$ since the maximum norm is the least $D$-norm and the sum norm is the largest $D$-norm. In addition to this, we obtain for a general logistic $D$-norm $\norm\cdot_\lambda$ with parameter $\lambda\in[1,\infty)$ by substituting $u\mapsto u^{1/\lambda}$
\begin{equation*}
\Cov(X,Y)=\int_0^{\infty}\frac{1}{(u^\lambda + 1)^{2/\lambda}}\,du-1=\frac{1}{\lambda}\int_0^{\infty}\frac{u^{1/\lambda-1}}{(u+ 1)^{2/\lambda}}\,du-1=\frac1{\lambda}B\left(\frac1\lambda,\frac1\lambda\right)-1,
\end{equation*}
where $B(x,y)=\int_0^1u^{x-1}(1-u)^{y-1}~du=\int_0^{\infty}\frac{u^{x-1}}{(1+u)^{x+y}}~du$ denotes the Euler beta function.
\end{exam}

In order to calculate the mean squared error of the predictor $\hat\eta_t$, we have to determine the mixed moment $E(\eta_t\hat\eta_t)$. For this purpose, we want to apply the previous lemma which is why we need to ensure that the vector $(\eta_t,\hat\eta_t)$ is standard max-stable itself. This is the content of the following Lemma.

\begin{lemma}\label{lem:original_predictor_smrv}
Let $\bm\eta=(\eta_t)_{t\in[0,1]}$ be an SMSP and denote by $\hat{\bm\eta}=(\hat\eta_t)_{t\in[0,1]}$ its discretized version with grid $\{s_0,\dotsc,s_d\}$. Then the bivariate rv $(\eta_t,\hat\eta_t)$ is standard max-stable for every $t\in[0,1]$ with  $D$-norm of the two-dimensional marginal
\begin{equation*}
\norm{(x,y)}_{D_t}:=\norm{\Big(x,g^\ast_{i-1}(t)y,g^\ast_i(t)y\Big)}_{D_{t,i-1,i}},\quad t\in[s_{i-1},s_i],~i=1,\dotsc,d,
\end{equation*}
where $\norm{\cdot}_{D_{t,i-1,i}}$ is the $D$-norm pertaining to $(\eta_t,\eta_{s_{i-1}},\eta_{s_i})$.
\end{lemma}

\begin{proof}
We have for every $t\in[s_{i-1},s_i]$, $x,y\leq 0$ and $i=1,\dotsc,d$
\begin{align*}
P(\eta_t\leq x,\hat\eta_t\leq y)&=P\left(\eta_t\leq x,\eta_{s_{i-1}}\leq g^\ast_{i-1}(t)y,\eta_{s_{i}}\leq g^\ast_i(t)y\right)\\
&=\exp\Big(-E\Big(\max\Big(\abs xZ_t,g^\ast_{i-1}(t)\abs yZ_{s_{i-1}},g^\ast_i(t)\abs yZ_{s_i}\Big)\Big)\Big)\\
&=\exp\Big(-E\Big(\max\Big(\abs xZ_t,\abs y\max\big(g^\ast_{i-1}(t)Z_{s_{i-1}},g^\ast_i(t)Z_{s_i}\big)\Big)\Big)\Big).
\end{align*}
The vector
\begin{equation*}
\Big(Z_t,\max\big(g^\ast_{i-1}(t)Z_{s_{i-1}},g^\ast_i(t)Z_{s_i}\big)\Big)
\end{equation*}
defines a generator for every $t\in[s_{i-1},s_i]$, $i=1,\dotsc,d$ as for all such $t$
\begin{equation*}
E\Big(\max\big(g^\ast_{i-1}(t)Z_{s_{i-1}},g^\ast_i(t)Z_{s_i}\big)\Big)=\norm{\left(g^\ast_{i-1}(t),g^\ast_i(t)\right)}_{D_{i-1,i}}=1.
\end{equation*}
\end{proof}

Let us recall the sequence of processes we have discussed in Theorem \ref{the:convergence_of_discretized}. Suppose $\bm\eta$ is an SMSP and choose a sequence of grids $\mathcal G_d$ of the interval $[0,1]$ with diameter $\kappa_d\to_{d\to\infty}0$. Denote by $\hat{\bm\eta}^{(d)}$, $d\in\N$, the sequence of dicretized versions of $\bm\eta$ with grid $\mathcal G_d$. Denote further by $\norm{\cdot}_{D^{(d)}_t}$ the $D$-norm pertaining to $(\eta_t,\hat\eta_t^{(d)})$, $t\in[0,1]$, $d\in\N$.

\begin{theorem}\label{the:mse_smsp}
Let $\bm\eta$ and $\hat{\bm\eta}^{(d)}$, $d\in\N$, be as above. The mean squared error of $\hat\eta_t^{(d)}$ is given by
\begin{equation*}
\MSE\left(\hat\eta_t^{(d)}\right):=E\left(\left(\eta_t-\hat\eta_t^{(d)}\right)^2\right)=2\left(2-\int_{0}^{\infty}\frac{1}{\norm{(1,u)}^2_{D_t^{(d)}}}~du\right)\to_{d\to\infty}0.
\end{equation*}
\end{theorem}

\begin{proof}
The second moment of a standard negative exponentially distributed random variable is
two and, therefore, we obtain by Lemma \ref{lem:covariance_bivariate_smrv} and Lemma \ref{lem:original_predictor_smrv}
\begin{equation*}
\MSE\left(\eta_t^{(d)}\right)=E\left(\eta_t^2\right)-2E\left(\eta_t\hat\eta_t^{(d)}\right)+E\left(\left(\hat\eta_t^{(d)}\right)^2\right)=4-2\int_{0}^{\infty}\frac{1}{\norm{(1,u)}^2_{D_t^{(d)}}}~du.
\end{equation*}
Next, we show $\norm\cdot_{D_t^{(d)}}\to_{d\to\infty}\norm\cdot_{\infty}$ for all $t\in[0,1]$. Denote by $\bm Z$ and $\hat{\bm Z}^{(d)}$, $d\in\N$, the generator processes of $\bm\eta$ and $\hat{\bm\eta}^{(d)}$, $d\in\N$. Define
\begin{equation*}
m:=E\left(\sup_{t\in[0,1]}Z_t\right)<\infty\quad\text{and}\quad\tilde Z:=\frac{\sup_{t\in[0,1]}Z_t}{m}.
\end{equation*}
Then $E(\tilde Z)=1$ and, thus, $(Z_t,\tilde Z)$ defines a generator for all $t\in[0,1]$. Denote by $\norm\cdot_{\tilde D}$ the $D$-norm pertaining to this generator. Lemma \ref{lem:structure_deterministic_functions} and Corollary \ref{cor:generation_of_smsp_discretized} imply $\hat{\bm Z}^{(d)}\leq \bm Z$ for all $d\in\N$. Therefore, we have for arbitrary $x,y\in\R$, $d\in\N$ and $t\in[0,1]$
\begin{equation*}
\max\left(\abs xZ_t,\abs y \hat Z_t^{(d)}\right)\leq \max\left(\abs xZ_t,\abs{m y} \tilde Z_t\right),
\end{equation*}
where
\begin{equation*}
E\left(\max\left(\abs xZ_t,\abs{m y} \tilde Z_t\right)\right)=\norm{(x,my)}_{\tilde D}<\infty.
\end{equation*}
Hence, the dominated convergence theorem, together with the fact that $\hat Z_t^{(d)}\to_{d\to\infty}Z_t$ for all $t\in[0,1]$ by Theorem \ref{the:convergence_of_discretized}, implies
\begin{align*}
\norm{(x,y)}_{D_t^{(d)}}=E\left(\max\left(\abs xZ_t,\abs y \hat Z_t^{(d)}\right)\right)&\to_{d\to\infty}E\left(\max\left(\abs xZ_t,\abs y Z_t\right)\right)\\
&=\norm{(x,y)}_{\infty}
\end{align*}
for all $x,y\in\R$.

In Example \ref{exam:covariances}, we have already calculated $\int_0^{\infty}\norm{(1,u)}_{\infty}^{-2}~du=2$. Since $\norm{\cdot}_{\infty}$ is the least $D$-norm, we have for all $d\in\N$ and $t\in[0,1]$
\begin{equation*}
\frac{1}{\norm{(1,u)}^2_{D_t^{(d)}}}\leq \frac{1}{\norm{(1,u)}^2_{\infty}},
\end{equation*}
and therefore by the dominated convergence theorem again
\begin{equation*}
\int_0^{\infty}\frac{1}{\norm{(1,u)}^2_{D_t^{(d)}}}~du\to_{d\to\infty}\int_0^{\infty}\frac{1}{\norm{(1,u)}^2_\infty}~du=2,
\end{equation*}
which completes the proof.
\end{proof}

%%%%%%%%%%%%%%%%%%%%%%%%%%%%%%%%%%%%%%%%%%%%%%%%%%%%%%%%%%
\section{Reconstruction of SGPP}\label{sec:predictGPP}
%%%%%%%%%%%%%%%%%%%%%%%%%%%%%%%%%%%%%%%%%%%%%%%%%%%%%%%%%%

The preceding technique of discretizing and reconstructing a given SMSP can also be applied to the case of SGPP by simply replacing the standard max-stable rv in the model \eqref{eq:generalized_max_linear_model} by a standard generalized Pareto distributed rv. Again, the generalized max-linear model results in an SGPP. Once this statement is proven, most of the results of the previous sections carry over in a very straightforward way.

Recall that a stochastic process $\bm V$ in $\bar C^-[0,1]$ is an SGPP, if there exists a $D$-norm $\norm\cdot_D$ on $E[0,1]$ and some $c_0>0$, such that $P(\bm V\leq f)=1-\norm f_D$ for all $f\in\bar E^-[0,1]$ with $\norm f_{\infty}\leq c_0$.

\subsection{Uniform Convergence of the Discretized Versions}

Let $(Y_0,Y_1,\dotsc,Y_d)$ follow a standard generalized Pareto distribution (GPD), i.\,e. there is a $D$-norm $\norm\cdot_{D_{0,\dotsc,d}}$ on $\R^{d+1}$ generated by $(Z_0,\dotsc,Z_d)$, and a vector $\bm y^{(0)}=(y_0^{(0)},\dotsc,y_d^{(0)})<\bm0$, such that $P(Y_0\leq y_0,\dotsc,Y_d\leq y_d)=1-\norm {\bm y}_{D_{0,\dotsc,d}}$ for all $\bm y=(y_0,\dotsc,y_d)$ with $\bm y^{(0)}\leq \bm y\leq\bm 0$. Note that this implies that each univariate marginal distribution of $(Y_0,\dotsc,Y_d)$ coincides in the upper tail with the uniform distribution on $[-1,0]$. For a detailed examination of GPD rv, see e.\,g. \citet{fahure10}.

We apply the generalized max-linear model to the GPD rv and obtain a stochastic process $\bm V=(V_t)_{t\in[0,1]}$,
\begin{equation}\label{eq:generalized_max_linear_model_sgpp}
V_t:=\max_{i=0,\dotsc,d}\frac{Y_i}{g_i(t)},
\end{equation}
where $g_0,\dotsc,g_d\in\bar C^+[0,1]$ are functions satisfying condition \eqref{eq:deterministic_functions}. Again, this process defines an SGPP as the next lemma shows.

\begin{lemma}\label{lem:max_linear_model_sgpp}
The stochastic process $\bm V=(V_t)_{t\in[0,1]}$ in \eqref{eq:generalized_max_linear_model_sgpp} defines an SGPP with generator process $\bm{\hat Z}=(\hat Z_t)_{t\in[0,1]}$,
\begin{equation*}
\hat Z_t=\max_{i=0,\dotsc,d}\left(g_i(t)Z_i\right),\quad t\in[0,1].
\end{equation*}
\end{lemma}

\begin{proof}
We have already shown in Lemma \ref{lem:max_linear_model_smsp} that $\bm{\hat Z}$ defines a generator process. Put $c_0:=-\max_{j=0,\dotsc,d}\left(y_j^{(0)}/\norm{g_j}_{\infty}\right)$. Then we have for all $f\in\bar E^-[0,1]$
\begin{equation*}
\norm f_{\infty}\leq c_0\iff \inf_{t\in[0,1]}f(t)\geq \max_{j=0,\dotsc,d}\left(y_j^{(0)}/\norm{g_j}_{\infty}\right)
\end{equation*}
and therefore for $i=0,\dotsc,d$
\begin{align*}
\inf_{t\in[0,1]}\left(g_i(t)f(t)\right)&\geq \inf_{t\in[0,1]}\left(g_i(t)\max_{j=0,\dotsc,d}\left(\frac{y_j^{(0)}}{\norm{g_j}_{\infty}}\right)\right)\\
&=\max_{j=0,\dotsc,d}\left(\frac{\inf_{t\in[0,1]}g_i(t)}{\sup_{t\in[0,1]}g_j(t)}\cdot y_j^{(0)}\right)\\
&\geq \frac{\inf_{t\in[0,1]}g_i(t)}{\sup_{t\in[0,1]}g_i(t)}\cdot y_i^{(0)}\\
&\geq y_i^{(0)}.
\end{align*}
Hence, we have for all $f$ close enough to zero
\begin{align*}
P(\bm V\leq f) &=P\left(Y_i\leq g_i(t)f(t),~i=0,\dotsc,d,~t\in[0,1]\right)\\
&=P\left(Y_i\leq \inf_{t\in[0,1]}\left(g_i(t)f(t)\right),~i=0,\dotsc,d\right)\\
%&=1-\norm{\left(\sup_{t\in[0,1]}\left(g_0\abs{f(t)}\right),\dotsc,\sup_{t\in[0,1]}\left(g_d(t)\abs{f(t)}\right)\right)}_{D_{0,\dotsc,d}}\\
%&=1-E\left(\max_{i=0,\dotsc,d}\left(\sup_{t\in[0,1]}\left(g_i(t)\abs{f(t)}\right)Z_i\right)\right)\\
%&=1-E\left(\sup_{t\in[0,1]}\left(\abs{f(t)}\max_{i=0,\dotsc,d}\left(g_i(t)Z_i\right)\right)\right)\\
&=1-E\left(\sup_{t\in[0,1]}\left(\abs{f(t)}Z'_t\right)\right),
\end{align*}
by the exact same arguments as in the proof of Lemma \ref{lem:max_linear_model_smsp}.
\end{proof}

Just like in the SMSP case, we can use this model to discretize and reconstruct a given SGPP. Let $\bm V=(V_t)_{t\in[0,1]}$ be an SGPP with generator process $\bfZ=(Z_t)_{t\in[0,1]}$ and $D$-norm $\norm\cdot_D$. Choose a grid $0=: s_0<s_1<\dots<s_{d-1}<s_d:=1$ of points within $[0,1]$. Then $\left(V_{s_0},\dots,V_{s_d}\right)$ is a standard GPD rv in $\R^{d+1}$ with pertaining $D$-norm $\norm\cdot_{D_{0,\dotsc,d}}$ generated by $\left(Z_{s_0},\dots,Z_{s_d}\right)$.

Now choose deterministic functions $g_0,\dotsc,g_d\in\bar C^+[0,1]$ with the property \eqref{eq:deterministic_functions} and put for $t\in[0,1]$
\begin{equation}\label{eq:generalized_max_linear_model_discretized_sgpp}
\hat V_t:=\max_{i=0,\dotsc,d}\frac {V_{s_i}}{g_i(t)}.
\end{equation}
As we have shown in Lemma \ref{lem:max_linear_model_sgpp}, the stochastic process $\hat{\bm V}=(\hat V_t)_{t\in[0,1]}$ in \eqref{eq:generalized_max_linear_model_discretized_sgpp} defines an SGPP with generator process $\hat{\bm Z}=(\hat Z_t)_{t\in[0,1]}$,
\begin{equation*}
\hat Z_t=\max_{i=0,\dotsc,d}\left(g_i(t)Z_{s_i}\right),\quad t\in[0,1].
\end{equation*}
By choosing the exact same functions $g^\ast_0,\dotsc,g^\ast_d$ as in the special case of the generalized max-linear model in Section \ref{sec:predictMSP}, we obtain the process
\begin{align}
\hat V_t&=\max\left(\frac{V_{s_{i-1}}}{g^\ast_{i-1}(t)},\frac{V_{s_{i}}}{g^\ast_{i}(t)}\right)\label{eq:discretizing_version_sgpp}\tag{\ref{eq:generalized_max_linear_model_discretized_sgpp}'}\\
&=\norm{(s_i-t,t-s_{i-1})}_{D_{i-1,i}}\max\left(\frac{V_{s_{i-1}}}{s_i-t},\frac{V_{s_i}}{t-s_{i-1}}\right),\quad t\in[s_{i-1},s_i],~i=1,\dotsc,d.\nonumber
\end{align}
In order to show that this process defines an SGPP as well, we only have to verify that the functions $g^\ast_0,\dotsc,g^\ast_d$ realize in $\bar C^+[0,1]$ and satisfy condition \eqref{eq:deterministic_functions}, which we have already done in Section \ref{sec:predictMSP}. Thus, the following result is proven.

\begin{cor}\label{cor:generation_of_sgpp_discretized}
Let $\bm V=(\eta_t)_{t\in[0,1]}$ be an SGPP with generator $\bm Z=(Z_t)_{t\in[0,1]}$, and $0:=s_0<s_1<,\dotsc,<s_{d-1}<s_d:=1$ be a grid in the interval $[0,1]$. The process $\hat{\bm V}=(\hat V_t)_{t\in[0,1]}$ defined in \eqref{eq:discretizing_version_sgpp} is an SGPP with generator process $\hat{\bm Z}=(\hat Z_t)_{t\in[0,1]}$, where
\begin{equation*}
\hat Z_t=\max\left(g^\ast_{i-1}(t)Z_{s_{i-1}},g^\ast_i(t)Z_{s_i}\right),\quad t\in[s_{i-1},s_i],~i=1,\dotsc,d.
\end{equation*}
The processes $\hat{\bm V}$ and $\hat{\bm Z}$ interpolate the rv $(V_{s_0},\dotsc,V_{s_d})$ and $(Z_{s_0},\dotsc,Z_{s_d})$, respectively.
\end{cor}

In complete accordance to the SMSP case we call $\hat{\bm V}$ the \emph{discretized version} of $\bm V$ with grid $\{s_0,\dotsc,s_d\}$.

\begin{rem}\upshape
Let the original SGPP $\bm V$ satisfy $P(\bm V\leq f)=1-\norm{f}_D$ for all $f\in\bar E^-[0,1]$ with $\norm f_{\infty}\leq c_0$, with some $D$-norm $\norm\cdot_D$ and some $c_0>0$. It is clear that this implies
\begin{equation*}
P(V_{s_0}\leq y_0,\dotsc,V_{s_d}\leq y_d)=1-\norm{\left(y_0,\dotsc,y_d\right)}_{D_{0,\dotsc,d}}
\end{equation*}
for all $\bm y:=(y_0,\dotsc,y_d)$ with $-c_0\leq y_i\leq 0$, $i=0,\dotsc,d$. Now denote by $\norm\cdot_{\hat D}$ the $D$-norm pertaining to the discretized version $\hat{\bm V}$ of $\bm V$ with grid $\{s_0,\dotsc,s_d\}$. Just like in the proof of Lemma \ref{lem:max_linear_model_sgpp} we obtain $P(\hat{\bm V}\leq f)=1-\norm f_{\hat D}$ for all $f\in\bar E^-[0,1]$ with $\norm f_{\infty}\leq \hat c_0$, where
\begin{equation*}
\hat c_0=-\max_{i=0,\dotsc,d}\left(\frac{-c_0}{\norm{g^\ast_i}_{\infty}}\right)=c_0,
\end{equation*}
since $\norm{g^\ast_i}_{\infty}=1$ holds for all $i=0,\dotsc,d$ according to Lemma \ref{lem:structure_deterministic_functions}. Thus, the upper tail where we know the distribution of the discretized version $\hat{\bm V}$ coincides with that of the initial SGPP $\bm V$.
\end{rem}

It is obvious that the pathwise structure of the discretized version of an SMSP we have established in Lemma \ref{lem:max_and_min_discretized} now carries over to the SGPP case since the assertion in this lemma  follows solely from the structure of $g^\ast_0,\dotsc,g^\ast_d$ and the fact that the initial process is nonpositive with probability one.
%\pagebreak

\begin{lemma}\label{lem:max_and_min_discretized_sgpp}
The SGPP defined in \eqref{eq:discretizing_version_sgpp} fulfills for $i=1,\dots,d$
\[
\sup_{t\in[s_{i-1},s_i]} \hat V_t=\max\left(V_{s_{i-1}},V_{s_i}\right),
\]
and
\begin{equation*}
\inf_{t\in[s_{i-1},s_i]} \hat V_t=-\norm{(V_{s_{i-1}},V_{s_i})}_{D_{s_{i-1},s_{i}}}.
\end{equation*}
This minimum is attained for $t=(s_{i-1}V_{s_{i-1}}+s_iV_{s_i})/V_{s_{i-1}}+V_{s_i})$.
\end{lemma}

Now consider a sequence of discretized versions $\hat{\bm V}^{(d)}$ of an SGPP $\bm V$ with grid $\mathcal G_d$, where the diameter of $\mathcal G_d$ converges to zero. Repeating the arguments in the proof of Theorem \ref{the:convergence_of_discretized_sgpp} yields the following result.

\begin{theorem}\label{the:convergence_of_discretized_sgpp}
The sequence of processes $\bm{\hat V}^{(d)}$, $d\in\N$, converges uniformly to $\bm V$ pathwise, i.\,e. $\norm{\bm{\hat V}^{(d)}-\bm V}_{\infty}\to_{d\to\infty}0$ with probability one.
\end{theorem}

\subsection{The Mean Squared Error of the Discretized Version}

The aim of this section is the calculation of the mean squared error of the predictor $\hat V_t$ of $V_t$. We obtain again some kind of pointwise convergence in mean square of a sequence of discretized versions with decreasing diameter to the initial SGPP. Nevertheless, there is a difference to the considerations in the previous section. In contrast to the case of max-stable distributions, we typically only know the distribution of a GPD rv in the upper tail. Note that the function $W(\bm x):=1-\norm{\bm x}_D$, $\bm x\leq 0$, $\norm{\bm x}_D\leq 1$, does not define a multivariate df in general, see cf. \citet{fahure10}. This fact forces us to consider \emph{conditional expectations} in this section.

In the bivariate case, however, $W$ defines a df, and we can assume that a GPD has this representation on the whole domain. The next Lemma is on some conditional moments of bivariate standard GPD rv in general. It can be shown by elementary computations similar to those in the proof of Lemma \ref{lem:covariance_bivariate_smrv}.

\begin{lemma}\label{lem:moments_bivariate_gpd}
Let $(U,V)$ be a standard GPD rv, i.\,e. there exists some $D$-norm $\norm\cdot_D$ such that $P(U\leq u,V\leq v)=1-\norm{(u,v)}_D$, $u,v\leq 0$, $\norm{(u,v)}_D\leq 1$. Then we have for all such $u,v$
\begin{enumerate}[(i)]
\item
\begin{equation*}
\quad P(U>u,V>v)=\norm{(u,v)}_1-\norm{(u,v)}_D,
\end{equation*}
\end{enumerate}
and, in case of $\norm\cdot_D\neq\norm\cdot_1$,
\begin{enumerate}[(i)]
\setcounter{enumi}{1}
\item
\begin{align*}
E(&U^2|U>u,V>v)\\
&=-\frac{\frac23u^3+u^2\left(u+\norm{(u,v)}_D\right)+v^3\int_0^{u/v}\int_0^{u/v}\norm{\left(\max(s,t),1\right)}_D~ds~dt}{\norm{(u,v)}_1-\norm{(u,v)}_D},\\
\end{align*}

\item
\begin{align*}
E(UV|U>u,V>v)=&-\frac{\int_v^0\int_u^0\norm{(s,t)}_D~ds~dt+v^3\int_0^{u/v}\norm{(r,1)}_D~dr}{\norm{(u,v)}_1-\norm{(u,v)}_D}\\
&-\frac{u^3\int_0^{v/u}\norm{(1,r)}_D~dr+uv\norm{(u,v)}_D}{\norm{(u,v)}_1-\norm{(u,v)}_D}.
\end{align*}
\end{enumerate}
\end{lemma}

Note that the case $\norm\cdot_D=\norm\cdot_1$ has to be treated with caution. It represents the case of uniform distribution on the line $\{(x,y):x,y\leq0,~x+y=-1\}$, which means that no observations fall in any rectangle $[u,0]\times[v,0]$, $u+v\geq -1$, cf. \citet{fahure10}.

\begin{exam}\label{exam:dependence_case}
In case of total dependence of $U$ and $V$ (i.\,e. $\norm\cdot_D=\norm\cdot_{\infty}$) and $u=v=:c$, the formulas in Lemma \ref{lem:moments_bivariate_gpd} (ii) and (ii) become
\begin{equation*}
E(U^2|U>u,V>v)=-\frac{\frac23c^3+c^2(c-c)+c^3}{-c}=\frac53c^2
\end{equation*}
and
\begin{equation*}
E(UV|U>c,V>c)=-\frac{-\int_c^0\int_c^0\min(s,t)~ds~dt+c^3+c^3-c^3}{-c}=\frac53c^2.
\end{equation*}
\end{exam}

We now return to the discretized Version $\hat{\bm V}$ of an SGPP $\bm V$. Again, we have to show that for every $t\in[0,1]$ the rv $(V_t,\hat V_t)$ follows a standard GPD. The exact same arguments as in Lemma \ref{lem:original_predictor_smrv} provide the bivariate df of this rv, at least in the upper tail.

\begin{lemma}\label{lem:original_predictor_gpd}
Let $\bm V=(V_t)_{t\in[0,1]}$ be an SGPP with generator $\bm Z=(Z_t)_{t\in[0,1]}$. Denote by $\hat{\bm V}=(\hat V_t)_{t\in[0,1]}$ its discretized version with grid $\{s_0,\dotsc,s_d\}$ and generator $\hat{\bm Z}=(\hat Z_t)_{t\in[0,1]}$. Then $(V_t,\hat V_t)$ defines a bivariate standard GPD rv for every $t\in[0,1]$. Its df is given by
\begin{align*}
P(V_t\leq x,\hat V_t&\leq y)=1-\norm{(x,y)}_{D_t}\\
&=1-\norm{\Big(x,g^\ast_{i-1}(t)y,g^\ast_i(t)y\Big)}_{D_{t,i-1,i}},\quad t\in[s_{i-1},s_i],~i=1,\dotsc,d,
\end{align*}
for $x,y$ close enough to zero, where $\norm{\cdot}_{D_{t}}$ is the $D$-norm gernerated by $(Z_t,\hat{Z}_t)$ and $\norm{\cdot}_{D_{t,i-1,i}}$ is the $D$-norm generated by $(Z_t,Z_{s_{i-1}},Z_{s_i})$.
\end{lemma}

We close this section with the calculation of the mean squared error of $\hat V_t$, under the condition that $V_t$ and $\hat V_t$ attain values that are close enough to zero, such that we have a representation of the df of $(V_t,\hat V_t)$ available in this area. Again, this means squared error converges to zero.

Suppose $\bm V$ is an SGPP  with generator $\bm Z$ and choose a sequence of grids $\mathcal G_d$ of the interval $[0,1]$ with fineness converging to zero as $d$ increases. Denote by $\hat{\bm V}^{(d)}$, $d\in\N$, the sequence of dicretized versions of $\bm V$ with grid $\mathcal G_d$, and by $\hat{\bm Z}^{(d)}$, $d\in\N$, their generators. Denote further by $\norm{\cdot}_{D^{(d)}_t}$ the $D$-norm generated by $(Z_t,\hat Z_t^{(d)})$, $t\in[0,1]$, $d\in\N$.

\begin{theorem}
Let $\bm V$ and $\hat{\bm V}^{(d)}$, $d\in\N$, be as above. Suppose $\norm{\cdot}_{D^{(d)}_t}\neq \norm\cdot_1$, $d\in\N$. Then we have for $c$ close enough to zero
\begin{equation*}
E\left(\left(V_t-\hat V_t^{(d)}\right)^2\Big|V_t>c,\hat V^{(d)}_t>c\right)\to_{d\to\infty}0.
\end{equation*}
\end{theorem}

\begin{proof}
According to Lemma \ref{lem:original_predictor_gpd}, the rv $\left(V_t,\hat V_t^{(d)}\right)$, $d\in\N$, is a standard GPD rv with pertaining $D$-norm $\norm\cdot_{D_t^{(d)}}$. We have already shown in the proof of Theorem \ref{the:mse_smsp} that $\norm\cdot_{D_t^{(d)}}\to_{d\to\infty}\norm\cdot_{\infty}$ pointwise. Substituting $\norm\cdot_D$ by $\norm\cdot_1$ in the numerators of Lemma \ref{lem:moments_bivariate_gpd} (ii) and (iii) leads to finite integrals exclusively, which is why we can apply the dominated convergence theorem in each of these integrals. Therefore, we obtain by the calculations in Example \ref{exam:dependence_case} for all $t\in[0,1]$
\begin{align*}
E&\left(\left(V_t-\hat V_t^{(d)}\right)^2\Big|V_t>c,\hat V_t^{(d)}>c\right)\\
&=E\left(V^2_t\big|V_t>c,\hat V_t^{(d)}>c\right)-2E\left(V_t\hat V_t^{(d)}\big|V_t>c,\hat V_t^{(d)}>c\right)\\
&\quad+E\left(\left(\hat V_t^{(d)}\right)^2\Big|V_t>c,\hat V_t^{(d)}>c\right)\\
&\to_{d\to\infty}\frac{10}{3}c^2-\frac{10}{3}c^2=0.
\end{align*}
\end{proof}

%%%%%%%%%%%%%%%%% Neu in der Revision

\section{Generalized Max-Linear Models in Arbitrary Dimension}\label{sec:arbitrary_dimension}
The preceding considerations can partially be extended from the function space $C[0,1]$ to arbitrary dimension $C([0,1]^m)$, $m\in\N$. Only an outline is presented here, details will be given in a separate paper. In dimension $m\ge 2$ we lose the natural ordering of the index space and, thus, the results will not be as precise as in dimension $m=1$ in general.

Let $K:[0,\infty)\to[0,1]$ be a continous and monotonically decreasing function (kernel) with the two properties
\begin{equation}\label{eq:condition_on_kernel}
K(0)=1,\qquad \lim_{t\to\infty}\frac{K(xt)}{K(yt)}=0,\quad 0\le y<x.
\end{equation}
The exponential kernel $K_e(t)=\exp(-t)$, $t\ge 0$, is a typical example. The convention $\frac00:=0$ in \eqref{eq:condition_on_kernel} further allows us to consider kernels with bounded support such as the triangular kernel $K_{\Delta}(t)=1-t$, $t\in[0,1]$, $K_{\Delta}(t)=0$, else.

Put for the bandwidth $h>0$
\[
K_h(\bft):= K\left(\frac{\norm{\bft}}h\right),\qquad \bft\in\R^m,
\]
where $\norm\cdot$ is an arbitrary norm on $\R^m$.

Let now $\bfs_1,\dots,\bfs_d$ be a grid of different points in $[0,1]^m$ and put for $i=1,\dots,d$ and $h>0$
\[
g^\ast_{i,h}(\bft):= \begin{cases}
0&,\,\mbox{if } K_h(\bft-\bfs_j)=0,\, 1\le j\le d,\\
\frac{K_h(\bft-\bfs_i)}{\norm{(K_h(\bft-\bfs_1),\dots,K_h(\bft-\bfs_d))}_D}&\mbox{elsewhere},
\end{cases}
\]
where $\norm\cdot_D$ is an arbitrary $D$-norm on $\R^d$.

Define for $i=1,\dots,d$
\[
N(\bfs_i):=\set{\bft\in[0,1]^m:\,\norm{\bft-\bfs_i}\le \norm{\bft-\bfs_j},\,j\not=i},
\]
which is the set of those points $\bft\in[0,1]^m$ that are closest to the grid point $\bfs_i$.

\begin{lemma}
We have for arbitrary $\bft\in[0,1]^m$ and $1\le i\le d$
\[
g^\ast_{i,h}(\bft)\to_{h\downarrow 0}\begin{cases}
1&,\mbox{ if }\bft=\bfs_i\\
0&,\mbox{ if }\bft\not\in N(\bfs_i)
\end{cases}
\]
as well as $g^\ast_{i,h}(\bft)\le 1$.
\end{lemma}

\begin{proof}
The convergence $g^\ast_{i,h}(\bm s_i)\to_{h\downarrow0}1$ follows from the fact that $K(0)=1$ and that the $D$-norm of a unit vector is 1.
The fact that an arbitrary $D$-norm is bounded below by the sup-norm together with the monotonicity of $K$ implies for $\bft\in[0,1]^m$
\begin{equation*}
g^\ast_{i,h}(\bft)\le \frac{K_h(\bft-\bfs_i)}{\max_{1\le j\le d}K_h(\bft-\bfs_j)}
=\frac{K\left(\frac{\norm{\bft-\bfs_i}}h\right)}{K\left(\frac{\min_{1\le j\le d}\norm{\bft-\bfs_j}}h\right)}\le 1.
\end{equation*}
Note that $K\left(\norm{\bft-\bfs_i}/h\right)/K\left(\min_{1\le j\le d}\norm{\bft-\bfs_j}/h\right) \to_{h\downarrow 0}0$
if $\bft\not\in N(\bfs_i)$ by the required growth condition on the kernel $K$ in \eqref{eq:condition_on_kernel}.
\end{proof}

 We, obviously, have
\begin{equation*}
\norm{(g^\ast_{1,h}(\bft),\dots,g^\ast_{d,h}(\bft))}_D=1
\end{equation*}
for all $\bft\in[0,1]^m$ such that $(g^\ast_{1,h}(\bft),\dots,g^\ast_{d,h}(\bft))\not=\bfzero$. This vector cannot vanish, for example, if $K$ is the exponential kernel.

Suppose now that $\bfeta=\left(\eta_{\bft}\right)_{\bft\in[0,1]^m}$ is a standard max-stable process on $[0,1]^m$, which we observe only through the grid of indices $\bfs_1,\dots,\bfs_d$. Put
\begin{equation*}
\tilde\eta_{\bft,h}:=\max_{i=1,\dots,d}\frac{\eta_{\bfs_i}}{g^\ast_{i,h}(\bft)},\qquad \bft\in[0,1]^m,
\end{equation*}
where the functions $g^\ast_{i,h}$ are defined as above via the $D$-norm corresponding to the rv $\left(\eta_{\bfs_1},\dots,\eta_{\bfs_d}\right)$.

If $(g^\ast_{1,h}(\bft),\dots,g^\ast_{d,h}(\bft))\not=\bfzero$ for all $\bft\in[0,1]^m$, then $\tilde\bfeta=\left(\tilde\eta_{\bft,h}\right)_{\bft\in[0,1]^m}$ is a standard max-stable process in $C\left([0,1]^m\right)$ since the continuity of $K$ implies the continuity of the paths of $\tilde\bfeta$ and each finite dimensional marginal distribution of $\tilde\bfeta$ is standard max-stable. Furthermore, in that case,
\begin{equation*}
\tilde\eta_{\bfs_j,h}\to_{h\downarrow 0}\eta_{\bfs_j},\qquad 1\le j\le d.
\end{equation*}
Further details of this approximation will be developed in a forthcoming paper.

\section*{Acknowledgement}
The authors are grateful to the reviewers of the first version of this paper for their careful reading. The paper has benefited a lot from their constructive remarks. The hints to a possible extension to arbitrary dimension initiated ongoing research.

%\bibliographystyle{enbib_arXiv}
%\bibliography{evt}
%\end{document}

\end{document}